\documentclass[11pt]{amsart}

\usepackage[T1]{fontenc}
\usepackage[utf8]{inputenc}
\usepackage{amsmath,amssymb,amsthm,mathtools}
\usepackage{mathrsfs}
\usepackage{tikz}
\usepackage{enumitem}
\usepackage{needspace}
\usepackage{microtype}
\usepackage{relsize}
\usepackage[hidelinks]{hyperref}
\usepackage[backend=biber,style=alphabetic,maxbibnames=15,maxcitenames=6,dateabbrev=false,autolang=hyphen]{biblatex}

\renewcommand{\UrlFont}{\sffamily\smaller}
\renewbibmacro{in:}{\ifentrytype{article}{}{\printtext{\bibstring{in}\intitlepunct}}}
\DeclareFieldFormat{url}{{\UrlFont\url{#1}}}
\DeclareFieldFormat{urldate}{%
  (version \thefield{urlday}\addspace%
  \mkbibmonth{\thefield{urlmonth}}\addspace%
  \thefield{urlyear}\isdot)}
\DeclareFieldFormat{eprint:arxiv}{%
\ifhyperref
    {\href{http://arxiv.org/abs/#1}{%
    arXiv\addcolon\nolinkurl{#1}}\iffieldundef{eprintclass}{}{\printtext{ [\thefield{eprintclass}]}}}
    {arXiv\addcolon\nolinkurl{#1}\iffieldundef{eprintclass}{}{\printtext{ [\thefield{eprintclass}]}}}}
\DeclareFieldFormat{eprint:arXiv}{%
\ifhyperref
    {\href{http://arxiv.org/abs/#1}{%
    arXiv\addcolon\nolinkurl{#1}}\iffieldundef{eprintclass}{}{\printtext{ [\thefield{eprintclass}]}}}
    {arXiv\addcolon\nolinkurl{#1}\iffieldundef{eprintclass}{}{\printtext{ [\thefield{eprintclass}]}}}}
\defbibheading{mgtbibliography}{\begin{center}\scshape References\end{center}}
\AtBeginBibliography{\small\setlength{\emergencystretch}{3em}}

\newtheorem{theorem}{Theorem}[section]
\newtheorem{lemma}[theorem]{Lemma}
\newtheorem{corollary}[theorem]{Corollary}
\newtheorem{proposition}[theorem]{Proposition}
\theoremstyle{definition}

\newtheorem{question}[theorem]{Question}
\theoremstyle{remark}

\newsavebox{\mgtpossiblebox}
\newsavebox{\mgtnecessarybox}
\sbox{\mgtpossiblebox}{\tikz[scale=.6ex/1cm,baseline=-.6ex,rotate=45,line width=.1ex]{\draw (-1,-1) rectangle (1,1);}}
\sbox{\mgtnecessarybox}{\tikz[scale=.6ex/1cm,baseline=-.6ex,line width=.1ex]{\draw (-1,-1) rectangle (1,1);}}
\DeclareMathOperator{\possible}{\text{\usebox{\mgtpossiblebox}}}
\DeclareMathOperator{\necessary}{\text{\usebox{\mgtnecessarybox}}}
\newcommand{\satisfies}{\models}
\newcommand{\theoryf}[1]{{\rm #1}}

\newcommand{\GrpEmb}{\mathrm{Grp}_{\hookrightarrow}}
\newcommand{\Lgrp}{\mathcal L_{\mathrm{Grp}}}
\newcommand{\LgrpModal}{\Lgrp^{\possible}}
\newcommand{\Val}{\mathrm{Val}}
\newcommand{\satisfiesctbl}{\satisfies_{\mathrm{ctbl}}}
\newcommand{\ord}{\operatorname{ord}}
\newcommand{\gen}[1]{\langle #1\rangle}

\newcommand{\code}{\operatorname{Code}}

\newcommand{\SFourTwo}{\ensuremath{\theoryf{S4.2}}}
\newcommand{\SFive}{\ensuremath{\theoryf{S5}}}
\newcommand{\Z}{\mathbb Z}
\newcommand{\N}{\mathbb N}
\newcommand{\Godel}{G\"odel}

\title{Modal group theory}
\author{Wojciech Aleksander Wo\l oszyn}
\address[Wojciech Aleksander Wo\l oszyn]
{Mathematical Institute, University of Oxford, Andrew Wiles Building, Radcliffe Observatory Quarter, Woodstock Road, Oxford, OX2 6GG, United Kingdom \&\ St Hilda's College, Cowley Place, Oxford, OX4 1DY, United Kingdom}
\email{wojciech@woloszyn.org}
\urladdr{https://woloszyn.org}

\begin{document}

\begin{abstract}
I introduce modal group theory, in which we study the category of all groups, considering embeddability as providing a notion of modal possibility. Using HNN extensions and Britton's lemma, I demonstrate that the modal language of groups is more expressive than the first-order language of groups. I interpret the theory of true arithmetic in modal group theory, and show that, as sets of \Godel\ numbers, it is computably isomorphic to the modal theory of finitely presented groups. I answer an open question of Berger, Block, and L\"owe~\cite{BBL23} by showing that the formulaic propositional modal validities of groups under embeddings are precisely \SFourTwo. I also analyze sentential validities and worlds validating \SFive.
\end{abstract}

\maketitle

\section{Introduction}

Modal group theory is about investigating the category of groups and embeddings from a modal perspective. From the group-theoretic point of view, this means studying what becomes visible when one is allowed to pass from a group to its overgroups. In modal group theory, for example, we want to know the expressive power of the modal language and exactly what the propositional modal validities in the category of groups are.

We say that an assertion $\varphi[\bar g]$ is \emph{possible} at a group $G$, and write $G\satisfies\possible\varphi[\bar g]$, if there is a group embedding $j\colon G\to H$ such that $H\satisfies\varphi[j(\bar g)]$. Similarly, $\varphi[\bar g]$ is \emph{necessary} at $G$ if for every embedding $j\colon G\to H$, we have $H\satisfies\varphi[j(\bar g)]$, in which case we write $G\satisfies\necessary\varphi[\bar g]$. Throughout this paper, possibility means embeddability; homomorphism and epimorphism versions are treated separately in \cite{WoloszynHomomorphisms,WoloszynEpimorphisms}.

To illustrate, every group is possibly necessarily nonabelian. That is, any group satisfies the assertion $\possible\necessary\exists x\exists y\,xy\neq yx$. Indeed, one can embed any group into a nonabelian group and it will remain nonabelian in all further groups. The statement $\exists x\exists y\,xy\neq yx$ is possible and once true, stays true forever. Therefore, the property of being nonabelian is a \emph{button}.

Similarly, it is necessarily possible that the center---the set of elements that commute with all the elements of the group---is trivial, because every group $G$ embeds into $G*F_2$, which has trivial center. But centerlessness, even once true, is not necessary, since from any centerless group $H$ one can pass to $H\times C_2$, whose center is nontrivial. Therefore, being centerless is a \emph{switch}, because we can turn it on and off in all further extensions.

\Needspace{20\baselineskip}
\begin{samepage}
I will show that:
\begin{enumerate}
    \item Many important group-theoretic properties inexpressible in the first-order language of groups become expressible in modal group theory. For example, equality of orders, membership in cyclic subgroups and in subgroups generated by a fixed finite tuple, cyclicity, being generated by at most a fixed number of elements, and torsion are expressible in modal group theory.
    \item As sets of \Godel\ numbers, the theory of true arithmetic is computably isomorphic to the modal theory of finitely presented groups.
    \item The formulaic propositional modal validities of groups under embeddings constitute precisely the modal logic \SFourTwo, answering an open question posed by Berger, Block, and L\"owe~\cite{BBL23}.
\end{enumerate}
\end{samepage}
I also analyze sentential validities and the modal validities holding at individual groups.

\medskip

Meanwhile, modal group theory is part of the larger \emph{modal model theory} program, which investigates modal principles and validities in Kripke categories generally, that is, in concrete categories of structures in a common first-order language. Hamkins and I introduced modal model theory, with modal graph theory as the first case study~\cite{HW24}. In my work on the modal theory of the category of sets, I develop the broader concrete-categorical perspective and extend this semantics beyond the original potentialist systems~\cite{WSet}; I treat the modal theory of linear orders in a separate preprint~\cite{WoloszynLinearOrders}. The terminology of potentialist systems, which can now be seen as a special case of Kripke categories, was introduced earlier in work on arithmetic and set-theoretic potentialism~\cite{HamkinsArithmeticPotentialism,HamkinsLinneboPotentialism}.

Finally, Ziegler's work on finite groups under embeddings contains a precursor to modal group theory; see also Hodges for a translation of part of this work~\cite{Ziegler80,Hodges84}.

\section*{Acknowledgments}
I thank Joel David Hamkins, Ehud Hrushovski, and Goh Jun Le for many helpful discussions related to this article.

\section{Expressive power of modal group theory}

We begin by showing that the modal language makes several natural group-theoretic notions definable. The formal group language is the inverse-free language $\Lgrp=\{\cdot,e\}$, and $\LgrpModal$ denotes its modal expansion by $\possible$. We take $\necessary\varphi$ to abbreviate $\neg\possible\neg\varphi$. Whenever inverse notation appears below, it is only metanotation: thus $u=x^{-1}$ abbreviates $xu=e\wedge ux=e$, and $s^{-1}xs=y$ abbreviates the inverse-free equation $xs=sy$.

\Needspace{20\baselineskip}
\begin{samepage}
\subsection{Expressing equality of order}

\begin{theorem}\label{thm:same-order}
The property that two elements have the same order is expressible in modal group theory. More precisely, there is a formula, denoted by $\ord(x)=\ord(y)$, such that for every group $G$,
\[
G\satisfies \ord(x)=\ord(y)
\quad\text{if and only if}\quad
x\text{ and }y\text{ have the same order in }G.
\]
\end{theorem}
\end{samepage}

Before the proof, recall the relevant HNN construction. Given a group $G$, isomorphic subgroups $A,B\leq G$, and an isomorphism $\phi\colon A\to B$, the HNN extension of $G$ with respect to $\phi$ is
\[
G^*=\langle G,s\mid s^{-1}as=\phi(a)\text{ for all }a\in A\rangle.
\]
The base group $G$ embeds into $G^*$, and $A$ and $B$ become conjugate in $G^*$ \cite{HNN49}.

\begin{proof}
We claim that $x$ and $y$ have the same order if and only if they can become conjugate in an extension. Accordingly, define
\[
\ord(x)=\ord(y)\quad\text{to abbreviate}\quad \possible\exists s\,(xs=sy).
\]
If $G\satisfies\possible\exists s\,(xs=sy)$, then there is an embedding $j\colon G\to H$ and an element $s\in H$ such that $j(x)s=sj(y)$. Equivalently, $s^{-1}j(x)s=j(y)$, so $j(x)$ and $j(y)$ are conjugate in $H$. Conjugate elements have the same order in $H$, and embeddings preserve order, so $x$ and $y$ have the same order in $G$.

Conversely, suppose that $x$ and $y$ have the same order in $G$. Let $A=\gen{x}$ and $B=\gen{y}$, and define an isomorphism $\phi\colon A\to B$ by $\phi(x^k)=y^k$. Form the HNN extension
\[
G^*=\langle G,s\mid s^{-1}as=\phi(a)\text{ for all }a\in A\rangle.
\]
Since $G$ embeds into $G^*$ and $s^{-1}xs=y$ holds in $G^*$ by construction, the displayed modal formula holds in $G$.
\end{proof}

\subsection{Expressing subgroup generation and cyclicity}

\begin{theorem}\label{thm:cyclic-subgroup-membership}
The property that an element belongs to the cyclic subgroup generated by another is expressible in modal group theory. There is a formula, denoted by $y\in\gen{x}$, such that for every group $G$,
\[
G\satisfies y\in\gen{x}
\quad\text{if and only if}\quad
y\in\gen{x}\text{ in }G.
\]
\end{theorem}

For the proof we use Britton's lemma for HNN extensions. If
\[
w=g_0s^{\varepsilon_1}g_1s^{\varepsilon_2}\cdots s^{\varepsilon_n}g_n
\]
with $g_i\in G$ and $\varepsilon_i\in\{\pm 1\}$, then $w$ is \emph{reduced} if it contains no subword of the form $s^{-1}as$ with $a\in A$ and no subword of the form $sbs^{-1}$ with $b\in B$. Britton's lemma says that every reduced word with $n\geq 1$ is nontrivial in the HNN extension \cite[p.~182]{LyndonSchupp}. In particular, in an HNN extension in which the stable letter $s$ centralizes a subgroup $A\leq G$, the elements of $G$ commuting with $s$ are exactly the elements of $A$.

\begin{proof}
We claim that $y\in\gen{x}$ if and only if it is necessary that every element commuting with $x$ also commutes with $y$. Accordingly, define
\[
y\in\gen{x}
\quad\text{to abbreviate}\quad
\necessary\forall t\,(tx=xt\to ty=yt).
\]
If $y=x^k$ for some $k\in\Z$, then every element commuting with $x$ commutes with every power of $x$, hence with $y$. Since this remains true in every extension, the displayed modal formula holds.

Conversely, suppose that
\[
G\satisfies \necessary\forall t\,(tx=xt\to ty=yt).
\]
Form the HNN extension
\[
G^*=\langle G,s\mid s\gamma=\gamma s\text{ for all }\gamma\in\gen{x}\rangle.
\]
By construction, $sx=xs$ holds in $G^*$. Since this necessary first-order assertion must hold in every extension of $G$, it holds in $G^*$, and instantiating $t=s$ yields $sy=ys$ in $G^*$.

By the consequence of Britton's lemma mentioned above, the elements of $G$ commuting with $s$ in $G^*$ are exactly the elements of $\gen{x}$. Hence $y\in\gen{x}$ in $G^*$. Cyclic subgroup membership for elements of the base group is absolute between $G$ and $G^*$, so $y\in\gen{x}$ already in $G$.
\end{proof}

\begin{corollary}\label{cor:cyclic}
The property of being a cyclic group is expressible in modal group theory.
\end{corollary}

\begin{proof}
By Theorem~\ref{thm:cyclic-subgroup-membership}, a group is cyclic exactly if it satisfies
\[
\exists x\,\forall y\,(y\in\gen{x}).
\]
\end{proof}

Every formula mentions only finitely many variables, so the same method extends from one generator to a finite tuple.

\begin{theorem}\label{thm:finitely-many-generators-membership}
For every $n\geq 1$, the property that an element belongs to the subgroup generated by a tuple of length $n$ is expressible in modal group theory. More precisely, there is a formula, denoted by $y\in\gen{x_1,\dots,x_n}$, such that for every group $G$,
\[
G\satisfies y\in\gen{x_1,\dots,x_n}
\quad\text{if and only if}\quad
y\in\gen{x_1,\dots,x_n}\text{ in }G.
\]
\end{theorem}

\begin{proof}
We claim that $y\in\gen{x_1,\dots,x_n}$ if and only if
\[
\necessary\forall t\Bigl(\bigwedge_{i=1}^n tx_i=x_it\to ty=yt\Bigr)
\]
holds. If $y$ is a word in the letters $x_i^{\pm 1}$, then any element commuting with each $x_i$ commutes with that word, and hence the displayed formula holds in every extension.

Conversely, assume the displayed formula holds in $G$, but $y\notin\gen{x_1,\dots,x_n}$. Form the HNN extension
\[
G^*=\langle G,s\mid s\gamma=\gamma s\text{ for all }\gamma\in\gen{x_1,\dots,x_n}\rangle.
\]
Then $s$ commutes with each $x_i$ in $G^*$. Since the displayed formula is necessary over $G$, it holds in $G^*$, so substituting $t=s$ yields $sy=ys$. Britton's lemma again implies that the elements of $G$ commuting with $s$ in $G^*$ are exactly the elements of $\gen{x_1,\dots,x_n}$, contradicting the assumption on $y$.
\end{proof}

\begin{corollary}\label{cor:n-generated}
For every integer $n\geq 1$, the property of being generated by at most $n$ elements is expressible in modal group theory.
\end{corollary}

\begin{proof}
By Theorem~\ref{thm:finitely-many-generators-membership}, a group is generated by at most $n$ elements exactly if it satisfies
\[
\exists x_1\cdots\exists x_n\,\forall y\,(y\in\gen{x_1,\dots,x_n}).
\]
\end{proof}

\begin{question}\label{q:finitely-generated}
Is the property of being finitely generated expressible in modal group theory?
\end{question}

\subsection{Expressing torsion}

\begin{theorem}\label{thm:torsion-element}
The property of being a torsion element is expressible in modal group theory. There is a formula, denoted by $\ord(x)<\infty$, such that for every group $G$,
\[
G\satisfies \ord(x)<\infty
\quad\text{if and only if}\quad
x\text{ has finite order in }G.
\]
\end{theorem}

\begin{proof}
Let $G$ be any group and let $x\in G$. We claim that $x$ has finite order if and only if either
\begin{enumerate}[label=\textup{(\arabic*)}]
    \item $\ord(x)\in\{1,2,3,4,6\}$, or
    \item there is an element $y\in G$ with $y\neq x$ and $xy\neq e$ such that $\gen{x}=\gen{y}$.
\end{enumerate}
Condition \textup{(1)} is first-order expressible. In a group, the equation $xy=e$ holds exactly when $y=x^{-1}$, so condition \textup{(2)} can be written in the inverse-free language as
\[
\exists y\bigl(y\neq x\wedge xy\neq e\wedge x\in\gen{y}\wedge y\in\gen{x}\bigr),
\]
using Theorem~\ref{thm:cyclic-subgroup-membership}.

Suppose first that $x$ has finite order $n\notin\{1,2,3,4,6\}$. Let $\varphi(n)$ be Euler's totient function. Since the only solutions of $\varphi(n)<3$ are $1,2,3,4,6$, there exists an integer $k$ relatively prime to $n$ such that $k\not\equiv \pm 1\pmod n$. Set $y=x^k$. Then $y\neq x^{\pm 1}$, and because $\gcd(k,n)=1$, B\'ezout's identity gives integers $a,b$ with $ak+bn=1$, so
\[
x=x^{ak+bn}=x^{ak}x^{bn}=x^{ak}=y^a.
\]
Thus $x\in\gen{y}$, while $y\in\gen{x}$ is immediate.

Conversely, suppose there is a $y\neq x^{\pm 1}$ with $\gen{x}=\gen{y}$. If $x$ had infinite order, then $\gen{x}\cong\Z$, whose only generators are $1$ and $-1$. Translating back to $\gen{x}$, the only generators would be $x$ and $x^{-1}$, contrary to the existence of $y$. Therefore $x$ has finite order.
\end{proof}

\begin{samepage}
\begin{corollary}\label{cor:torsion}
\leavevmode\par\nobreak
\begin{enumerate}[label=\textup{(\arabic*)},topsep=0.25em]
    \item The set of torsion elements is definable in modal group theory.
    \item The property of being a torsion group is expressible in modal group theory.
\end{enumerate}
\end{corollary}
\end{samepage}

\begin{proof}
Item \textup{(1)} is immediate from Theorem~\ref{thm:torsion-element}. Item \textup{(2)} is expressed by
\[
\forall x\,(\ord(x)<\infty).
\]
\end{proof}

\section{Non-expressivity in first-order group theory}

The modal definitions above genuinely go beyond first-order group theory.

\begin{theorem}\label{thm:not-first-order}
None of the following relations or properties are first-order expressible in the ordinary language of groups:
\begin{enumerate}[label=\textup{(\arabic*)}]
    \item two elements have the same order;
    \item for any fixed $n\geq 1$, an element belongs to the subgroup generated by an $n$-tuple;
    \item for any fixed $n\geq 1$, a group is generated by at most $n$ elements; in particular, a group is cyclic;
    \item an element is torsion;
    \item a group is torsion.
\end{enumerate}
Moreover, by item \textup{(4)}, the set of torsion elements is not first-order definable in general.
\end{theorem}

\begin{proof}
For item \textup{(1)}, consider the cyclic groups $C_{2p_n}$, where $p_n$ is the $n$th prime. Let $x_n$ be a generator and $y_n=x_n^2$. Then $x_n$ has order $2p_n$ and $y_n$ has order $p_n$, so the two orders are different in every factor. But in any nonprincipal ultraproduct $\prod_n C_{2p_n}/U$, both $[x_n]_U$ and $[y_n]_U$ have infinite order, hence the same order. By Łoś's theorem~\cite{ChangKeisler}, equality of order is not first-order expressible.

For item \textup{(2)}, fix $n\geq 1$ and suppose there were a first-order formula $\phi(y,x_1,\dots,x_n)$ expressing $y\in\gen{x_1,\dots,x_n}$. In the language of groups with constants $c,d_1,\dots,d_n$, consider the theory consisting of group theory, the sentence $\phi(c,d_1,\dots,d_n)$, and the sentences
\[
c\neq w(d_1,\dots,d_n)
\]
for every group word $w$ in $n$ letters. This theory is not satisfiable. But every finite fragment is satisfiable in the free group $F_n=\gen{a_1,\dots,a_n}$: choose a reduced word $w_0(a_1,\dots,a_n)$ representing an element of $F_n$ different from the finitely many elements represented by the forbidden words, interpret $d_i$ by $a_i$ and $c$ by $w_0(a_1,\dots,a_n)$, and note that $c\in\gen{d_1,\dots,d_n}$. By compactness, no such first-order formula exists.

For item \textup{(3)}, fix $n\geq 1$ and suppose being generated by at most $n$ elements were first-order expressible. Then the corresponding first-order theory would have the free group $F_n$ as an infinite model, and hence, by the upward Löwenheim--Skolem theorem, would have uncountable models. But every finitely generated group is countable. So being generated by at most $n$ elements is not first-order expressible. The case $n=1$ gives cyclicity.

For item \textup{(4)}, suppose there were a first-order formula $\tau(x)$ expressing that $x$ is torsion. In the language of groups with a constant $c$, consider the theory consisting of group theory, the sentence $\tau(c)$, and the sentences $c^n\neq e$ for every $n\geq 1$. This theory is inconsistent, but every finite fragment is realized in a sufficiently large finite cyclic group, interpreting $c$ as a generator. This contradicts compactness, so torsion of an element is not first-order expressible.

For item \textup{(5)}, let $G_n$ be finite cyclic groups of strictly increasing orders. Each $G_n$ is a torsion group. But any nonprincipal ultraproduct $\prod_n G_n/U$ contains an element of infinite order, namely the class of generators. Hence the ultraproduct is not torsion. By Łoś's theorem, torsion is not first-order expressible.
\end{proof}

\begin{corollary}\label{cor:compactness-LS-fail}
Modal group theory satisfies neither compactness nor the upward L\"owenheim--Skolem property.
\end{corollary}

\begin{proof}
By Theorem~\ref{thm:torsion-element}, $\ord(x)<\infty$ expresses torsion. Hence, in the modal language with one additional constant $c$, the theory
\[
\{\ord(c)<\infty\}\cup\{c^n\neq e:n\geq 1\}
\]
is finitely satisfiable, as every finite fragment is realized in a sufficiently large finite cyclic group with $c$ interpreted as a generator, but it is not satisfiable. Thus compactness fails.

By Corollary~\ref{cor:cyclic}, the modal sentence $\exists x\,\forall y\,(y\in\gen{x})$ expresses cyclicity. It has the infinite model $\Z$, but no uncountable models, since every cyclic group is countable. Thus the upward L\"owenheim--Skolem property fails.
\end{proof}

\section{The complexity of modal group theory}

We now exploit the expressive power established above to interpret arithmetic.

\subsection{Arithmetic with an infinite-order parameter}

\begin{lemma}\label{lem:presburger-with-parameter}
Presburger arithmetic with divisibility over the integers is interpretable in modal group theory with a parameter of infinite order. More precisely, if $G$ is a group containing an element $g$ of infinite order, then there is a computable translation $\varphi\mapsto\varphi^*(x)$ from sentences of the language $\{+,\mid,0,1\}$ to modal group-theoretic formulas such that
\[
\langle\Z,+,\mid,0,1\rangle\satisfies\varphi
\quad\text{if and only if}\quad
G\satisfies\varphi^*[g].
\]
\end{lemma}

\begin{proof}
Fix $g\in G$ of infinite order. We represent the integer $i\in\Z$ by the group element $g^i$. Thus $0$ is represented by $e$ and $1$ by $g$. Addition is interpreted by group multiplication:
\[
i+j=k\quad\text{if and only if}\quad g^ig^j=g^k.
\]
Divisibility is interpreted by cyclic subgroup membership:
\[
i\mid j\quad\text{if and only if}\quad g^j\in\gen{g^i}.
\]
By Theorem~\ref{thm:cyclic-subgroup-membership}, the latter relation is expressible in modal group theory.

By a standard effective flattening procedure, every sentence of the language $\{+,\mid,0,1\}$ can be transformed computably into an equivalent sentence without compound terms, in which every atomic formula is one of the forms
\[
u=v,\qquad u=0,\qquad u=1,\qquad u+v=w,\qquad u\mid v,
\]
with $u,v,w$ variables. We therefore define the translation on formulas in this normal form and, for an arbitrary input sentence, first replace it by its flattened equivalent.

Let $\code_x(y)$ abbreviate the modal formula $y\in\gen{x}$. We translate atomic formulas by
\[
\begin{aligned}
(u=v)^*&:=u=v, & (u=0)^*&:=u=e,\\
(u=1)^*&:=u=x, & (u+v=w)^*&:=uv=w,\\
(u\mid v)^*&:=v\in\gen{u}.
\end{aligned}
\]
Boolean connectives are translated recursively in the obvious way, and quantifiers are restricted to the coded copy of $\Z$:
\[
(\exists y\,\psi)^*:=\exists y\,(\code_x(y)\wedge\psi^*),\qquad
(\forall y\,\psi)^*:=\forall y\,(\code_x(y)\to\psi^*).
\]
Since $x=g$ has infinite order, the set of its powers is naturally isomorphic to $\Z$, and under this identification the translated formulas express exactly the intended arithmetic relations. The flattening step and the recursive translation are both computable, so the resulting interpretation is computable.
\end{proof}

The next two lemmas are elementary but useful. To avoid any ambiguity, the formal formulas exhibited below are written entirely in the language $\{+,\mid,0,1\}$; subtraction and multiplication notation are used only as informal metanotation in the surrounding discussion.

\begin{lemma}\label{lem:squaring}
The squaring function is definable in $\langle\Z,+,\mid,0,1\rangle$. More precisely, there is a formula $\phi(x,y)$ such that
\[
\langle\Z,+,\mid,0,1\rangle\satisfies\phi[x,y]
\quad\text{if and only if}\quad
y=x^2.
\]
\end{lemma}

\begin{proof}
Consider the formula
\[
\begin{aligned}
\phi(x,y):=\exists u\exists t\Bigl(&
u+1=x \wedge x+t=y \wedge {}\\
&\forall z\Bigl(
(((x\mid z)\wedge ((x+1)\mid z))\leftrightarrow ((y+x)\mid z))\\
&\qquad\wedge
(((x\mid z)\wedge (u\mid z))\leftrightarrow (t\mid z))
\Bigr)\Bigr).
\end{aligned}
\]
This is a formula of the language $\{+,\mid,0,1\}$. We claim that it defines squaring.

If $y=x^2$, let $u=x-1$ and $t=y-x=x(x-1)$. Since consecutive integers are relatively prime, the common multiples of $x$ and $x+1$ are exactly the multiples of $x(x+1)=y+x$, while the common multiples of $x$ and $u=x-1$ are exactly the multiples of $x(x-1)=t$. Hence $\phi(x,y)$ holds.

Conversely, suppose $\phi(x,y)$ holds, witnessed by $u$ and $t$. Then $u=x-1$ and $t=y-x$. The first biconditional says that $y+x$ has the same set of multiples as $x(x+1)$, and the second says that $y-x$ has the same set of multiples as $x(x-1)$. Therefore there are signs $\varepsilon,\delta\in\{\pm1\}$ such that
\[
y+x=\varepsilon x(x+1),\qquad y-x=\delta x(x-1).
\]
If $x=0$, then the two equations immediately yield $y=0=x^2$. Assume now that $x\neq 0$. Subtracting the displayed equations gives
\[
2x=\varepsilon x(x+1)-\delta x(x-1),
\]
and dividing by $x$ yields
\[
2=\varepsilon(x+1)-\delta(x-1).
\]
If $(\varepsilon,\delta)=(1,1)$, this identity is automatic. If $(\varepsilon,\delta)=(1,-1)$, it forces $x=1$. If $(\varepsilon,\delta)=(-1,1)$, it forces $x=-1$. The choice $(\varepsilon,\delta)=(-1,-1)$ is impossible. In each of the three viable cases, substituting back into the displayed equations yields $y=x^2$. Thus $\phi$ defines the squaring function.
\end{proof}

\begin{lemma}\label{lem:multiplication}
Multiplication is definable in $\langle\Z,+,\mid,0,1\rangle$.
\end{lemma}

\begin{proof}
Let $\phi_{\square}(u,v)$ be the formula from Lemma~\ref{lem:squaring}, so that $\phi_{\square}(u,v)$ defines $v=u^2$. Consider the formula
\[
\begin{aligned}
\mu(x,y,z):=\exists a\exists b\exists c\exists d\Bigl(&
\phi_{\square}(x+y,c)\wedge \phi_{\square}(x,a)\wedge {}\\
&\phi_{\square}(y,b)\wedge a+b=d \wedge d+(z+z)=c
\Bigr).
\end{aligned}
\]
This is a formula of the language $\{+,\mid,0,1\}$. If $z=xy$, choose $a=x^2$, $b=y^2$, $c=(x+y)^2$, and $d=x^2+y^2$; then the last conjunct is exactly the identity
\[
(x+y)^2=x^2+2xy+y^2.
\]
Conversely, if $\mu(x,y,z)$ holds, then for suitable witnesses $a,b,c,d$ we have $a=x^2$, $b=y^2$, $c=(x+y)^2$, and $d=a+b$, while $d+(z+z)=c$. Hence
\[
z+z=(x+y)^2-x^2-y^2=2xy,
\]
and therefore $z=xy$. Thus multiplication is definable in $\langle\Z,+,\mid,0,1\rangle$.
\end{proof}

\begin{lemma}\label{lem:ring-with-parameter}
The theory of the ring of integers is interpretable in modal group theory with a parameter of infinite order. More precisely, if $G$ contains an element $g$ of infinite order, then there is a computable injective translation $\varphi\mapsto\varphi^*(x)$ from sentences of the ring language $\{+,\cdot,0,1\}$ such that
\[
\langle\Z,+,\cdot,0,1\rangle\satisfies\varphi
\quad\text{if and only if}\quad
G\satisfies\varphi^*[g].
\]
\end{lemma}

\begin{proof}
By Lemma~\ref{lem:presburger-with-parameter}, modal group theory with a parameter of infinite order interprets $\langle\Z,+,\mid,0,1\rangle$. By Lemma~\ref{lem:multiplication}, multiplication is definable in that structure. Hence the full ring of integers is interpretable. Fix once and for all a recursive syntactic implementation $T(\varphi)(x)$ of this interpretation.

To make the translation injective on G\"odel codes, add harmless syntactic tags. For each $n\in\N$, let $\theta_n(x)$ be the right-associated conjunction of $n+1$ copies of the tautology $x=x$. The formulas $\theta_n(x)$ are computable, pairwise syntactically distinct, and true under every assignment. Define
\[
\varphi^*(x):=\theta_{\code(\varphi)}(x)\wedge T(\varphi)(x).
\]
The tag does not change truth, while it syntactically records the G\"odel code of $\varphi$. Therefore $\varphi\mapsto\varphi^*(x)$ is computable and injective on G\"odel codes, and the displayed equivalence follows from the correctness of the underlying interpretation $T$.
\end{proof}

\subsection{The ring of integers and true arithmetic}

\begin{theorem}\label{thm:arithmetic-in-modal-group-theory}
The theory of the ring of integers is interpretable in modal group theory. There is a computable injective translation $\varphi\mapsto\varphi^{\dagger}$ from sentences of the ring language $\{+,\cdot,0,1\}$ to modal group-theoretic sentences such that, for every group $G$,
\[
\langle\Z,+,\cdot,0,1\rangle\satisfies\varphi
\quad\text{if and only if}\quad
G\satisfies\varphi^{\dagger}.
\]
Consequently, $\varphi^{\dagger}$ holds at all groups if and only if it holds at the trivial group.
\end{theorem}

\begin{proof}
By Theorem~\ref{thm:torsion-element}, the property of having infinite order is expressible as $\neg(\ord(x)<\infty)$. By Theorem~\ref{thm:cyclic-subgroup-membership}, we can express that an element is a power of another. Hence we can eliminate the distinguished infinite-order parameter from Lemma~\ref{lem:ring-with-parameter} by quantifying for it existentially inside the scope of a possibility operator.

Define $\varphi^{\dagger}$ by
\[
\varphi^{\dagger}:=\possible\exists x\,(\neg(\ord(x)<\infty)\wedge \varphi^*(x)),
\]
where $\varphi^*(x)$ is the parameter-dependent translation from Lemma~\ref{lem:ring-with-parameter}. Suppose first that $\langle\Z,+,\cdot,0,1\rangle\satisfies\varphi$, and let $G$ be any group. By Lemma~\ref{lem:ring-with-parameter}, the infinite cyclic group $\Z$ with parameter $1$ satisfies $\varphi^*(1)$. Embed $G$ into $G\times\Z$ by $g\mapsto (g,0)$, and let $x:=(e_G,1)$. Then $x$ has infinite order in $G\times\Z$, and again by Lemma~\ref{lem:ring-with-parameter} we have $(G\times\Z)\satisfies\varphi^*(x)$. Hence $G\satisfies\varphi^{\dagger}$.

Conversely, suppose that $G\satisfies\varphi^{\dagger}$. Then there is a group $H$, an embedding $j\colon G\to H$, and an element $x\in H$ such that $x$ has infinite order and $H\satisfies\varphi^*(x)$. By Lemma~\ref{lem:ring-with-parameter}, it follows that $\langle\Z,+,\cdot,0,1\rangle\satisfies\varphi$.

By Lemma~\ref{lem:ring-with-parameter}, the map $\varphi\mapsto\varphi^*(x)$ is computable and injective on G\"odel codes. Adjoining the fixed outer wrapper $\possible\exists x\,(\neg(\ord(x)<\infty)\wedge \cdot)$ therefore yields a computable injective translation $\varphi\mapsto\varphi^{\dagger}$. The final equivalence follows by taking $G$ to be arbitrary or trivial.
\end{proof}

Recall that a set $A$ is \emph{one--one reducible} to a set $B$ if there is a total injective computable function $f:\N\to\N$ such that $n\in A$ exactly when $f(n)\in B$. A \emph{computable isomorphism} between $A$ and $B$ is a computable bijection $p:\N\to\N$ such that $n\in A$ exactly when $p(n)\in B$.

All theories below are regarded as subsets of $\N$ via a fixed G\"odel coding of sentences; codes that are not sentences of the relevant language are not elements of the theory. We shall use the following totalization convention for the sentence translations constructed below. Suppose that $\tau$ is a computable translation from source-language sentences to target-language sentences, injective on sentence codes, and truth-preserving for the source and target theories under consideration. Before using $\tau$ as a one--one reduction, replace it harmlessly by the truth-equivalent tagged translation
\[
\tau^\sharp(\sigma):=(\exists z\,z=z)\wedge \tau(\sigma),
\]
with the displayed parenthesization fixed as part of the coding. Let $\chi_n$ be the right-associated conjunction of $n+1$ copies of the true sentence $\exists z\,z=z$, and set
\[
\delta_n:=(\forall z\,z\neq z)\wedge \chi_n.
\]
The sentences $\delta_n$ are pairwise syntactically distinct target-language contradictions, and their displayed outer form is disjoint from the image of $\tau^\sharp$. Define a total function on all natural numbers by
\[
F_\tau(n)=
\begin{cases}
\code(\tau^\sharp(\sigma)), & \text{if } n=\code(\sigma) \text{ for a source-language sentence }\sigma,\\
\code(\delta_n), & \text{if } n \text{ is not a source-language sentence code.}
\end{cases}
\]
The predicate ``$n$ codes a source-language sentence'' is decidable, so $F_\tau$ is total computable. It is injective by the injectivity of $\tau$ on sentence codes, the pairwise distinctness of the $\delta_n$, and the disjoint outer tags. Moreover, non-sentence codes are outside the source theory, while each $\delta_n$ is outside the target theory. Thus every such sentence translation induces a genuine one--one reduction of sets of natural numbers. We use Myhill's isomorphism theorem in the form that if $A\leq_1 B$ and $B\leq_1 A$, then there is a computable isomorphism between $A$ and $B$ \cite{Myhill55}.

By \emph{true arithmetic} we mean the complete first-order theory of $\langle\N,+,\cdot,0,1\rangle$.

\begin{lemma}\label{lem:true-arithmetic-to-Z}
There is a computable injective translation $\sigma\mapsto\sigma^{\Z}$ from sentences of first-order arithmetic to sentences of the ring language $\{+,\cdot,0,1\}$ such that
\[
\langle\N,+,\cdot,0,1\rangle\satisfies\sigma
\quad\text{if and only if}\quad
\langle\Z,+,\cdot,0,1\rangle\satisfies\sigma^{\Z}.
\]
\end{lemma}

\begin{proof}
Inside $\langle\Z,+,\cdot,0,1\rangle$ define
\[
\mathrm{Nat}(x):=\exists a\exists b\exists c\exists d\,\bigl(x=a^2+b^2+c^2+d^2\bigr).
\]
By Lagrange's four-square theorem, $\mathrm{Nat}(x)$ holds exactly of the nonnegative integers. Let $\sigma\mapsto\sigma^{\circ}$ be the usual relativized translation: atomic formulas of first-order arithmetic are translated verbatim into the ring language, Boolean connectives are translated recursively, and quantifiers are relativized to $\mathrm{Nat}(x)$:
\[
(\exists x\,\psi)^{\circ}:=\exists x\,(\mathrm{Nat}(x)\wedge \psi^{\circ}),\qquad
(\forall x\,\psi)^{\circ}:=\forall x\,(\mathrm{Nat}(x)\to\psi^{\circ}).
\]
Since addition and multiplication on the definable domain $\mathrm{Nat}(\Z)$ agree with addition and multiplication on $\N$, we have
\[
\langle\N,+,\cdot,0,1\rangle\satisfies\sigma
\quad\text{if and only if}\quad
\langle\Z,+,\cdot,0,1\rangle\satisfies\sigma^{\circ}.
\]

To make the sentence map injective, fix a computable family $(\alpha_n)_{n\in\N}$ of pairwise syntactically distinct true ring sentences; for definiteness, let $\alpha_n$ be the sentence $\overline{n}=\overline{n}$. Define
\[
\sigma^{\Z}:=\alpha_{\code(\sigma)}\wedge \sigma^{\circ}.
\]
The added tag is true in $\langle\Z,+,\cdot,0,1\rangle$, so it preserves the displayed equivalence. The map $\sigma\mapsto\sigma^{\Z}$ is computable and injective on G\"odel codes because the left conjunct syntactically records the code of $\sigma$.
\end{proof}

\begin{corollary}\label{cor:true-arithmetic-reduction}
True arithmetic is one--one reducible to modal group theory.
\end{corollary}

\begin{proof}
Let $f$ be the sentence map sending a first-order arithmetic sentence $\sigma$ to
\[
f(\sigma):=(\sigma^{\Z})^{\dagger}.
\]
By Lemma~\ref{lem:true-arithmetic-to-Z} and Theorem~\ref{thm:arithmetic-in-modal-group-theory},
\[
\begin{aligned}
\langle\N,+,\cdot,0,1\rangle\satisfies\sigma
&\quad\text{if and only if}\quad
\langle\Z,+,\cdot,0,1\rangle\satisfies\sigma^{\Z}\\
&\quad\text{if and only if}\quad
f(\sigma)\text{ holds at all groups}.
\end{aligned}
\]
The sentence map $f$ is computable and injective, since both $\sigma\mapsto\sigma^{\Z}$ and $\varphi\mapsto\varphi^{\dagger}$ are. By the totalization convention above, it therefore induces a total computable injective function on $\N$ preserving membership and non-membership. Thus true arithmetic is one--one reducible to modal group theory.
\end{proof}

\subsection{Finitely presented groups}

Let $\satisfies_{\mathrm{fp}}$ denote truth in the category of finitely presented groups and embeddings between them. Throughout, a word on $n$ generators is coded by the G\"odel number of its letter-sequence. If $e$ codes a finite presentation, write $\mathrm{Elt}(e,w)$ for the decidable predicate asserting that $w$ is a word on the generators named in the presentation coded by $e$.

\begin{theorem}\label{thm:fp-arithmetization}
For every code $e$ of a finitely presented group $G_e$, every tuple of word-codes $\bar w=\langle w_1,\dots,w_k\rangle$ of the same length as $\bar x$ such that
\[
\bigwedge_{i=1}^k \mathrm{Elt}(e,w_i),
\]
and every formula $\varphi(\bar x)$ in the modal language of groups, there is a computable arithmetical formula $\varphi^{\mathrm{fp}}(e,\bar w)$ such that
\[
(G_e,\bar w)\satisfies_{\mathrm{fp}}\varphi
\quad\text{if and only if}\quad
\N\satisfies\varphi^{\mathrm{fp}}(e,\bar w).
\]
Consequently, the modal theory of finitely presented groups is uniformly interpretable in true arithmetic.
\end{theorem}

Here each word-code $w_i$ names the element of $G_e$ represented by that word. Thus $(G_e,\bar w)$ denotes the expansion of $G_e$ by constants interpreted as the elements represented by $w_1,\dots,w_k$.

\begin{proof}
A finite presentation
\[
G=\langle a_1,\dots,a_n\mid r_1,\dots,r_m\rangle
\]
is coded by the triple $(n,m,\langle r_1,\dots,r_m\rangle)$. The set of such codes is decidable, so we may define an arithmetical predicate $\mathrm{FPGroup}(e)$ expressing that $e$ is a code of a finite presentation.

An embedding code is a tuple $h=\langle u_1,\dots,u_n\rangle$ of word-codes in a second presentation $e'$, representing the homomorphism that sends the $i$th generator of $e$ to the word $u_i$ in $G_{e'}$. Let $\mathrm{Tuple}_n(h)$ assert that $h$ codes an $n$-tuple, and write $(h)_i$ for its $i$th component. Let $\operatorname{ap}(h,u)$ be the primitive-recursive word-substitution function obtained by replacing the $i$th generator of the source presentation by the word $(h)_i$ and replacing formal inverses accordingly. Let $\mathrm{Apply}(h,u,u')$ abbreviate the equation
\[
u'=\operatorname{ap}(h,u).
\]
Let $\mathrm{Hom}(e,h,e')$ assert that $h$ induces a well-defined homomorphism $G_e\to G_{e'}$: namely, $\mathrm{Tuple}_n(h)$ holds, every $(h)_i$ is a word in the presentation $e'$, and every relator of $e$ maps to the identity in $G_{e'}$. This is a $\Sigma^0_1$ predicate. Whenever $\mathrm{Hom}(e,h,e')$ holds and $\mathrm{Elt}(e,u)$ holds, the word $\operatorname{ap}(h,u)$ is a word in the presentation $e'$.

Write $\mathrm{Triv}(e,w)$ for the $\Sigma^0_1$ predicate asserting that there exists a derivation of $w=1$ in $G_e$. For word-codes $u,v$ in the presentation coded by $e$, write
\[
u\equiv_e v
\]
to mean that $u$ and $v$ represent the same element of $G_e$, that is,
\[
u\equiv_e v
\quad\text{if and only if}\quad
\mathrm{Elt}(e,u)\wedge \mathrm{Elt}(e,v)\wedge \mathrm{Triv}(e,uv^{-1}).
\]
Now define $\mathrm{Emb}(e,h,e')$ to assert that $h$ induces a homomorphism with trivial kernel:
\[
\begin{aligned}
\mathrm{Emb}(e,h,e') := {}&\mathrm{Hom}(e,h,e')\wedge {}\\
&\forall u\forall u'\Bigl(\mathrm{Elt}(e,u)\wedge \mathrm{Apply}(h,u,u')\wedge \mathrm{Elt}(e',u')\\
&\hspace{5.5em}\wedge \mathrm{Triv}(e',u')\to \mathrm{Triv}(e,u)\Bigr).
\end{aligned}
\]

For each formula $\psi$ and assignment $\bar w$, define recursively an arithmetical predicate $\mathrm{Val}_{\psi}(e,\bar w)$. If $\psi$ is atomic of the form $t_1=t_2$, compute the word-code $v=v(t_1,t_2,\bar w)$ for the word obtained by substituting the assignment into $t_1t_2^{-1}$, and set
\[
\mathrm{Val}_{t_1=t_2}(e,\bar w):=\mathrm{Triv}(e,v).
\]
Boolean connectives are handled in the obvious way. First-order quantifiers range over word-codes satisfying $\mathrm{Elt}(e,w)$; explicitly,
\[
\begin{aligned}
\mathrm{Val}_{\exists y\,\theta}(e,\bar w)&:=
\exists z\,\bigl(\mathrm{Elt}(e,z)\wedge \mathrm{Val}_{\theta}(e,\bar w,z)\bigr),\\
\mathrm{Val}_{\forall y\,\theta}(e,\bar w)&:=
\forall z\,\bigl(\mathrm{Elt}(e,z)\to \mathrm{Val}_{\theta}(e,\bar w,z)\bigr).
\end{aligned}
\]
The representative-independence claim below justifies quantifying over word-codes rather than over elements themselves. Since $\necessary\psi$ is, by definition, the abbreviation $\neg\possible\neg\psi$, it suffices to give the modal clause for $\possible$.

For the modal case, if $\bar w=\langle w_1,\dots,w_k\rangle$, define
\[
\begin{aligned}
\mathrm{Val}_{\possible\psi}(e,\bar w):={}&\exists e'\exists h\exists w_1'\cdots\exists w_k'\Bigl(\mathrm{FPGroup}(e')\wedge \mathrm{Emb}(e,h,e')\wedge {}\\
&\bigwedge_{i=1}^k\mathrm{Apply}(h,w_i,w_i')\wedge \bigwedge_{i=1}^k\mathrm{Elt}(e',w_i')\\
&\wedge \mathrm{Val}_{\psi}(e',\langle w_1',\dots,w_k'\rangle)\Bigr).
\end{aligned}
\]

We shall use the following elementary observation. If $\mathrm{Hom}(e,h,e')$ holds, $u\equiv_e v$, $\mathrm{Apply}(h,u,u')$, and $\mathrm{Apply}(h,v,v')$, then
\[
u'\equiv_{e'} v'.
\]
Indeed, a derivation of $uv^{-1}=1$ from the relators of $e$ may be substituted by the word-substitution coded by $h$. The images of the relators of $e$ are trivial in $G_{e'}$ by $\mathrm{Hom}(e,h,e')$, so splicing in derivations of those image relators gives a derivation of $u'(v')^{-1}=1$ in $G_{e'}$.

\smallskip
\noindent\textbf{Claim.}
For every finite-presentation code $e$ and every modal formula $\psi(\bar x)$, if $\bar u=\langle u_1,\dots,u_k\rangle$ and $\bar v=\langle v_1,\dots,v_k\rangle$ are tuples of word-codes satisfying
\[
\bigwedge_{i=1}^k u_i\equiv_e v_i,
\]
then
\[
\mathrm{Val}_{\psi}(e,\bar u)
\quad\text{if and only if}\quad
\mathrm{Val}_{\psi}(e,\bar v).
\]

\smallskip
\noindent\emph{Proof of claim.}
The proof is by induction on the complexity of $\psi$, uniformly in the finite-presentation code $e$.

For atomic formulas $t_1=t_2$, group terms respect $\equiv_e$. Hence, if $u_i\equiv_e v_i$ for every $i$, then the word obtained by substituting $\bar u$ into $t_1t_2^{-1}$ represents the identity in $G_e$ if and only if the word obtained by substituting $\bar v$ into $t_1t_2^{-1}$ represents the identity in $G_e$.

Boolean connectives are immediate.

For first-order quantifiers, use the same quantified word-code on both sides. For instance, if
\[
\mathrm{Val}_{\exists y\,\theta}(e,\bar u)
\]
is witnessed by a word-code $z$ with $\mathrm{Elt}(e,z)$, then $z\equiv_e z$, and the induction hypothesis applied to the tuples $(\bar u,z)$ and $(\bar v,z)$ yields
\[
\mathrm{Val}_{\theta}(e,\bar v,z).
\]
Thus $\mathrm{Val}_{\exists y\,\theta}(e,\bar v)$. The converse is symmetric. The universal quantifier is handled similarly, or by the definition of $\forall$ from $\exists$ and negation.

It remains to consider the modal case. Suppose
\[
\mathrm{Val}_{\possible\theta}(e,\bar u)
\]
holds. Then there are $e'$, $h$, and word-codes $\bar u'=\langle u_1',\dots,u_k'\rangle$ such that
\[
\mathrm{FPGroup}(e')\wedge \mathrm{Emb}(e,h,e')
\]
holds, each $u_i'$ is the image of $u_i$ under $h$, and
\[
\mathrm{Val}_{\theta}(e',\bar u')
\]
holds. For each $i$, let $v_i'=\operatorname{ap}(h,v_i)$. Since $\mathrm{Emb}(e,h,e')$ implies $\mathrm{Hom}(e,h,e')$, and since $u_i\equiv_e v_i$, the observation above gives
\[
u_i'\equiv_{e'} v_i'
\]
for every $i$. Also $\mathrm{Elt}(e',v_i')$ holds, since $\mathrm{Hom}(e,h,e')$ and $\mathrm{Elt}(e,v_i)$ hold. Since $\mathrm{FPGroup}(e')$ holds, the induction hypothesis may be applied in the target presentation $e'$, giving
\[
\mathrm{Val}_{\theta}(e',\bar u')
\quad\text{if and only if}\quad
\mathrm{Val}_{\theta}(e',\bar v').
\]
Therefore $\mathrm{Val}_{\theta}(e',\bar v')$ holds. The same $e'$ and $h$, together with the image tuple $\bar v'$, witness
\[
\mathrm{Val}_{\possible\theta}(e,\bar v).
\]
The reverse implication is symmetric. This proves the claim.

\smallskip
Set $\varphi^{\mathrm{fp}}(e,\bar w):=\mathrm{Val}_{\varphi}(e,\bar w)$. The map $\varphi\mapsto\varphi^{\mathrm{fp}}$ is primitive-recursive.

We prove by induction on $\varphi$ that
\[
(G_e,\bar w)\satisfies_{\mathrm{fp}}\varphi
\quad\text{if and only if}\quad
\N\satisfies\varphi^{\mathrm{fp}}(e,\bar w).
\]
The atomic and Boolean steps are straightforward. For the quantifier step, one uses the representative-independence claim: every element of $G_e$ has a word representative, and the truth of $\mathrm{Val}_{\psi}(e,\bar w)$ is unchanged when $\bar w$ is replaced by any tuple of word-codes representing the same elements.

For the modal step, it suffices to consider $\possible\psi$, since $\necessary\psi$ is by definition the abbreviation $\neg\possible\neg\psi$. Suppose first that $(G_e,\bar w)\satisfies_{\mathrm{fp}}\possible\psi$. Then there is an embedding $j:G_e\hookrightarrow G_{e'}$ into a finitely presented group such that $(G_{e'},j(\bar w))\satisfies_{\mathrm{fp}}\psi$. Choose word representatives in $G_{e'}$ for the images under $j$ of the generators of $G_e$; these give a code $h$ with $\mathrm{Emb}(e,h,e')$. If $w_i'=\operatorname{ap}(h,w_i)$, then $w_i'$ represents $j(w_i)$. By the induction hypothesis, $\mathrm{Val}_{\psi}(e',\bar w')$ holds, and so $\mathrm{Val}_{\possible\psi}(e,\bar w)$ holds.

Conversely, suppose $\mathrm{Val}_{\possible\psi}(e,\bar w)$ holds. Then there are $e'$, $h$, and $\bar w'$ witnessing the displayed arithmetical modal clause. Since $\mathrm{Emb}(e,h,e')$ holds, $h$ induces an injective homomorphism $j_h:G_e\hookrightarrow G_{e'}$. The equations $\mathrm{Apply}(h,w_i,w_i')$ say that $w_i'$ represents $j_h(w_i)$. By the induction hypothesis, $(G_{e'},\bar w')\satisfies_{\mathrm{fp}}\psi$. Hence $(G_e,\bar w)\satisfies_{\mathrm{fp}}\possible\psi$.

Now suppose that $\varphi$ is a sentence. Define
\[
\widehat{\varphi}_{\mathrm{fp}}:=\forall e\,(\mathrm{FPGroup}(e)\to \varphi^{\mathrm{fp}}(e)).
\]
Then
\[
\varphi\text{ is valid in all finitely presented groups}
\quad\text{if and only if}\quad
\N\satisfies\widehat{\varphi}_{\mathrm{fp}}.
\]
To obtain an injective sentence map, fix a computable family $(\tau_n)_{n\in\N}$ of pairwise distinct true arithmetic sentences; for definiteness, let $\tau_n$ be the sentence $\overline{n}=\overline{n}$, where $\overline{n}$ is the usual numeral for $n$. Now define
\[
\widetilde{\varphi}_{\mathrm{fp}}:=\tau_{\code(\varphi)}\wedge \widehat{\varphi}_{\mathrm{fp}}.
\]
Since each $\tau_n$ is true in $\N$, we still have
\[
\varphi\text{ is valid in all finitely presented groups}
\quad\text{if and only if}\quad
\N\satisfies\widetilde{\varphi}_{\mathrm{fp}}.
\]
The map $\varphi\mapsto\widetilde{\varphi}_{\mathrm{fp}}$ is computable. It is also injective on G\"odel codes of sentences, because $\varphi\mapsto \code(\varphi)$ is computable, the family $(\tau_n)_{n\in\N}$ is pairwise syntactically distinct, and adjoining the fixed right conjunct $\widehat{\varphi}_{\mathrm{fp}}$ preserves that distinction. By the totalization convention above, this sentence map induces a genuine total one--one reduction on subsets of $\N$. This yields both the desired uniform interpretation of the modal theory of finitely presented groups in true arithmetic and a one--one reduction of that theory to true arithmetic.
\end{proof}

\begin{lemma}\label{lem:cyclic-membership-fp}
Let $G$ be a finitely presented group and let $x,y\in G$. Then the modal formula from Theorem~\ref{thm:cyclic-subgroup-membership}, when interpreted in the finitely presented frame, still defines cyclic subgroup membership:
\[
G\satisfies_{\mathrm{fp}} y\in\gen{x}
\quad\text{if and only if}\quad
y\in\gen{x}\text{ in }G.
\]
\end{lemma}

\begin{proof}
If $y\in\gen{x}$ in $G$, then every element commuting with $x$ in any extension of $G$ commutes with $y$, and hence in particular in every finitely presented overgroup of $G$. So the formula from Theorem~\ref{thm:cyclic-subgroup-membership} holds in the finitely presented frame.

Conversely, assume that $G\satisfies_{\mathrm{fp}} y\in\gen{x}$, where the displayed formula is interpreted in the finitely presented frame. Form the HNN extension
\[
G^*=\langle G,s\mid sx=xs\rangle.
\]
Because $G$ is finitely presented and the associated subgroup $\gen{x}$ is cyclic, this HNN extension is finitely presented: if
\[
G=\langle a_1,\dots,a_n\mid r_1,\dots,r_m\rangle
\]
and $x$ is represented by a word $w(a_1,\dots,a_n)$, then
\[
G^*=\langle a_1,\dots,a_n,s\mid r_1,\dots,r_m,\, sw=ws\rangle.
\]
Thus $G^*$ is a finitely presented overgroup of $G$. Since $sx=xs$ holds in $G^*$ and the necessary assertion from Theorem~\ref{thm:cyclic-subgroup-membership} holds over all finitely presented overgroups, substituting $t=s$ yields $sy=ys$ in $G^*$. By the same Britton-lemma argument as in Theorem~\ref{thm:cyclic-subgroup-membership}, the elements of $G$ commuting with $s$ in $G^*$ are exactly the elements of $\gen{x}$. Hence $y\in\gen{x}$ in $G$.
\end{proof}

\begin{corollary}\label{cor:torsion-fp}
Let $G$ be a finitely presented group and let $x\in G$. Then the formula from Theorem~\ref{thm:torsion-element}, interpreted in the finitely presented frame, still defines finite order:
\[
G\satisfies_{\mathrm{fp}} \ord(x)<\infty
\quad\text{if and only if}\quad
x\text{ has finite order in }G.
\]
\end{corollary}

\begin{proof}
The formula from Theorem~\ref{thm:torsion-element} is a first-order combination of equalities together with instances of cyclic subgroup membership. By Lemma~\ref{lem:cyclic-membership-fp}, those cyclic-subgroup clauses have the intended meaning in the finitely presented frame. Therefore the same proof as in Theorem~\ref{thm:torsion-element} gives the displayed equivalence.
\end{proof}

\begin{lemma}\label{lem:ring-with-parameter-fp}
Let $G$ be a finitely presented group and let $g\in G$ have infinite order. Then the translation $\varphi\mapsto\varphi^*(x)$ from Lemma~\ref{lem:ring-with-parameter} satisfies
\[
\langle\Z,+,\cdot,0,1\rangle\satisfies\varphi
\quad\text{if and only if}\quad
G\satisfies_{\mathrm{fp}}\varphi^*[g]
\]
for every sentence $\varphi$ of the ring language $\{+,\cdot,0,1\}$.
\end{lemma}

\begin{proof}
In the translation from Lemma~\ref{lem:presburger-with-parameter}, modal cyclic subgroup membership occurs in two places: in the coded-domain predicate
\[
\code_x(y)\equiv y\in\gen{x},
\]
which is used to relativize quantifiers, and in the divisibility clause
\[
(u\mid v)^*:=v\in\gen{u}.
\]
By Lemma~\ref{lem:cyclic-membership-fp}, both uses of cyclic subgroup membership have the same meaning in the finitely presented frame as in the full embedding frame, whenever the ambient group is finitely presented. Therefore, when $G$ is finitely presented and $g\in G$ has infinite order, the quantifier relativization ranges over exactly the coded copy $\gen{g}$ of $\Z$, and the divisibility atom has its intended arithmetic meaning. Hence the same inductive proof as in Lemma~\ref{lem:presburger-with-parameter} shows that
\[
\langle\Z,+,\mid,0,1\rangle\satisfies\psi
\quad\text{if and only if}\quad
G\satisfies_{\mathrm{fp}}\psi^*[g]
\]
for every sentence $\psi$ of the language $\{+,\mid,0,1\}$. Since multiplication is definable in $\langle\Z,+,\mid,0,1\rangle$ by Lemma~\ref{lem:multiplication}, the same recursive syntactic implementation used in Lemma~\ref{lem:ring-with-parameter} yields the displayed equivalence for ring-language sentences.
\end{proof}

\begin{corollary}\label{cor:arithmetic-in-fp-modal-group-theory}
For every finitely presented group $G$ and every sentence $\varphi$ of the ring language $\{+,\cdot,0,1\}$,
\[
\langle\Z,+,\cdot,0,1\rangle\satisfies\varphi
\quad\text{if and only if}\quad
G\satisfies_{\mathrm{fp}}\varphi^{\dagger}.
\]
\end{corollary}

\begin{proof}
Suppose first that $\langle\Z,+,\cdot,0,1\rangle\satisfies\varphi$, and let $G$ be finitely presented. Then $G\times\Z$ is finitely presented and $G$ embeds into $G\times\Z$ by $g\mapsto (g,0)$. Let $x:=(e_G,1)\in G\times\Z$. By Corollary~\ref{cor:torsion-fp}, the element $x$ satisfies $\neg(\ord(x)<\infty)$ in the finitely presented frame, and by Lemma~\ref{lem:ring-with-parameter-fp} we have $(G\times\Z)\satisfies_{\mathrm{fp}}\varphi^*(x)$. Hence $G\satisfies_{\mathrm{fp}}\varphi^{\dagger}$.

Conversely, suppose that $G\satisfies_{\mathrm{fp}}\varphi^{\dagger}$. Then there is a finitely presented group $H$, an embedding $j\colon G\to H$, and an element $x\in H$ such that
\[
H\satisfies_{\mathrm{fp}}\neg(\ord(x)<\infty)\wedge \varphi^*(x).
\]
By Corollary~\ref{cor:torsion-fp}, the element $x$ has infinite order in $H$. Lemma~\ref{lem:ring-with-parameter-fp} therefore yields $\langle\Z,+,\cdot,0,1\rangle\satisfies\varphi$.
\end{proof}

\begin{theorem}\label{thm:fp-computable-isomorphism}
The modal theory of finitely presented groups is computably isomorphic to true arithmetic.
\end{theorem}

\begin{proof}
By Theorem~\ref{thm:fp-arithmetization} and the totalized form of the computable injective sentence map
\[
\varphi\mapsto \widetilde{\varphi}_{\mathrm{fp}}
\]
constructed in its proof, the modal theory of finitely presented groups is one--one reducible to true arithmetic.

For the converse reduction, fix a computable injective family $(\theta_n)_{n\in\N}$ of group-theoretic sentences valid in every group; for example,
\[
\theta_n:=\exists x_0\cdots\exists x_n\,\bigwedge_{i=0}^n (x_i=x_i).
\]
Now send a first-order arithmetic sentence $\sigma$ to
\[
f(\sigma):=\theta_{\code(\sigma)}\wedge (\sigma^{\Z})^{\dagger},
\]
where $\sigma^{\Z}$ is the computable injective translation from Lemma~\ref{lem:true-arithmetic-to-Z}. The map $f$ is computable and injective because the left conjunct syntactically records the G\"odel code of $\sigma$.

Let $G$ be any finitely presented group. By Lemma~\ref{lem:true-arithmetic-to-Z} and Corollary~\ref{cor:arithmetic-in-fp-modal-group-theory},
\[
\begin{aligned}
\langle\N,+,\cdot,0,1\rangle\satisfies\sigma
&\quad\text{if and only if}\quad
\langle\Z,+,\cdot,0,1\rangle\satisfies\sigma^{\Z}\\
&\quad\text{if and only if}\quad
G\satisfies_{\mathrm{fp}} (\sigma^{\Z})^{\dagger}.
\end{aligned}
\]
Since each $\theta_n$ is valid in every group, we therefore have
\[
\langle\N,+,\cdot,0,1\rangle\satisfies\sigma
\quad\text{if and only if}\quad
G\satisfies_{\mathrm{fp}} f(\sigma).
\]
As this equivalence holds for every finitely presented group $G$, it follows that
\[
\begin{aligned}
\langle\N,+,\cdot,0,1\rangle\satisfies\sigma
&\quad\text{if and only if}\quad
f(\sigma)\text{ is valid}\\
&\quad\text{in all finitely presented groups}.
\end{aligned}
\]
By the totalization convention above, this sentence map induces a genuine total one--one reduction of true arithmetic to the modal theory of finitely presented groups. Since the converse one--one reduction was established at the beginning of the proof, Myhill's isomorphism theorem in the form recalled above now yields a computable isomorphism.
\end{proof}

\subsection{Countable groups}

Throughout this subsection, modal truth is interpreted in the potentialist system of countable groups and embeddings between countable groups. Accordingly, if $G$ codes a countable group and $\bar a\in D_G^k$, then
\[
(G,\bar a)\satisfiesctbl\varphi
\]
means that $\varphi$ holds at the coded group under the assignment $\bar a$ when the modal operator $\possible$ quantifies only over embeddings into countable overgroups, equivalently over coded countable groups $H\subseteq\N$.

We code a countable group by a single set $G\subseteq\N$ using tagged pairs. Fix a primitive-recursive pairing function $\langle\cdot,\cdot\rangle$, and write
\[
\begin{aligned}
D_G(x) &\quad\text{if and only if}\quad \langle 0,x\rangle\in G,\\
M_G(x,y,z) &\quad\text{if and only if}\quad
\langle 1,\langle x,\langle y,z\rangle\rangle\rangle\in G.
\end{aligned}
\]
Thus $D_G$ is the coded domain and $M_G$ is the coded multiplication graph.

\begin{theorem}\label{thm:countable-groups-second-order}
For every countable group-code $G\subseteq\N$, every tuple $\bar a$ of natural numbers from the coded domain $D_G$, and every formula $\varphi(\bar x)$ in the modal language of groups, there is a primitive-recursive second-order arithmetic formula $\varphi^{\ddagger}(G,\bar a)$ such that
\[
(G,\bar a)\satisfiesctbl\varphi
\quad\text{if and only if}\quad
\langle\N,+,\cdot,0,1,\mathcal P(\N)\rangle\satisfies\varphi^{\ddagger}(G,\bar a).
\]
Consequently, the modal theory of countable groups under embeddings between countable groups is computably interpretable in true second-order arithmetic.
\end{theorem}

\begin{proof}
Any countable group $K$ can be coded in this way: choose an injection $i\colon K\to\N$ with $i(1_K)=0$, let $D_G$ be the range of $i$, and let $M_G(i(a),i(b),i(c))$ hold exactly when $ab=c$ in $K$. Conversely, every set $G\subseteq\N$ satisfying the corresponding second-order axioms determines a countable group with domain $D_G$ and multiplication $M_G$.

A second-order formula $\mathrm{Group}(G)$ asserts that $0\in D_G$, that $M_G$ is total and functional on $D_G$, and that it satisfies the group axioms with identity $0$ and inverses in $D_G$.

Given group-codes $G,H\subseteq\N$, an embedding from the coded group of $G$ into the coded group of $H$ is coded by a set $F\subseteq\N$, where $\langle x,y\rangle\in F$ means $f(x)=y$. A second-order formula $\mathrm{Emb}(G,F,H)$ says that $F$ is the graph of a total injective homomorphism from $D_G$ into $D_H$.

For each group term $t(\bar x)$ we define, by recursion on term complexity, a second-order arithmetic formula $\mathrm{Eval}_{t}(G,\bar x,y)$ asserting that $y$ is the value of $t$ in the group coded by $G$ under the assignment $\bar x\in D_G^k$. The defining clauses are
\[
\mathrm{Eval}_{x_i}(G,\bar x,y):=(D_G(y)\wedge y=x_i),\qquad \mathrm{Eval}_{e}(G,\bar x,y):=(y=0),
\]
and, for composite terms,
\[
\mathrm{Eval}_{u\cdot v}(G,\bar x,y):=\exists y_0\exists y_1\bigl(\mathrm{Eval}_{u}(G,\bar x,y_0)\wedge \mathrm{Eval}_{v}(G,\bar x,y_1)\wedge M_G(y_0,y_1,y)\bigr).
\]
The map $t\mapsto\mathrm{Eval}_{t}$ is primitive-recursive on G\"odel codes of terms.

Now define, by structural induction on formulas, second-order arithmetic formulas $\mathrm{Val}_{\psi}(G,\bar x)$. For an atomic equation $t_1=t_2$, let
\[
\mathrm{Val}_{t_1=t_2}(G,\bar x):=\exists y_1\exists y_2\bigl(\mathrm{Eval}_{t_1}(G,\bar x,y_1)\wedge \mathrm{Eval}_{t_2}(G,\bar x,y_2)\wedge y_1=y_2\bigr).
\]
Boolean connectives are translated directly, and first-order quantifiers are relativized to the coded domain:
\[
\mathrm{Val}_{\exists u\,\psi}(G,\bar x):=\exists u\bigl(D_G(u)\wedge \mathrm{Val}_{\psi}(G,\bar x,u)\bigr),
\]
\[
\mathrm{Val}_{\forall u\,\psi}(G,\bar x):=\forall u\bigl(D_G(u)\to \mathrm{Val}_{\psi}(G,\bar x,u)\bigr).
\]
Here the right-hand sides use the extended assignment tuple obtained by appending $u$ to $\bar x$. Since $\necessary\psi$ is, by definition, the abbreviation $\neg\possible\neg\psi$, it suffices to give the modal clause for $\possible$. If the free variables of $\psi$ are among $\bar x=(x_1,\dots,x_k)$, then for the modal operator set
\[
\begin{aligned}
\mathrm{Val}_{\possible\psi}(G,\bar x):={}&\exists H\exists F\exists y_1\cdots\exists y_k\Bigl(\mathrm{Group}(H)\wedge \mathrm{Emb}(G,F,H)\\
&\wedge \bigwedge_{i=1}^k \langle x_i,y_i\rangle\in F \wedge \mathrm{Val}_{\psi}(H,\bar y)\Bigr).
\end{aligned}
\]
where $\bar y=(y_1,\dots,y_k)$. Because $\mathrm{Emb}(G,F,H)$ says that $F$ is the graph of a total injective homomorphism from $D_G$ into $D_H$, the tuple $\bar y$ is exactly the image tuple of $\bar x$. Every clause in the recursive definition is primitive-recursive on G\"odel codes, so $\psi\mapsto\mathrm{Val}_{\psi}$ is primitive-recursive.

Set $\varphi^{\ddagger}(G,\bar a):=\mathrm{Val}_{\varphi}(G,\bar a)$. An induction on group terms shows that $\mathrm{Eval}_{t}(G,\bar a,b)$ holds exactly when $b$ is the value of $t$ in the countable group coded by $G$ under the assignment $\bar a\in D_G^k$. Using this, a second induction on the complexity of $\varphi$ yields the desired equivalence. The atomic and Boolean steps are immediate. The quantifier step uses the relativization to $D_G$. For the modal step, it suffices to consider $\possible\psi$, since $\necessary\psi$ is by definition the abbreviation $\neg\possible\neg\psi$. The witnesses $H$, $F$, and $\bar y$ decode to an actual embedding of coded countable groups carrying $\bar a$ to its image tuple inside a countable overgroup, and conversely any witnessing embedding in the countable frame yields suitable second-order parameters.

If $\varphi$ is a sentence, define
\[
\widehat\varphi:=\forall G\,(\mathrm{Group}(G)\to \varphi^{\ddagger}(G)).
\]
Then
\[
\varphi\text{ is valid in all countable groups}
\quad\text{if and only if}\quad
\langle\N,+,\cdot,0,1,\mathcal P(\N)\rangle\satisfies\widehat\varphi.
\]
This already gives the desired computable interpretation of the modal theory of countable groups under embeddings between countable groups in true second-order arithmetic. To obtain a one--one reduction, fix a computable family $(\eta_n)_{n\in\N}$ of pairwise syntactically distinct true second-order arithmetic sentences; for definiteness, let $\eta_n$ be the right-associated conjunction of $n+1$ copies of the true atomic sentence $0=0$. Now define
\[
\widetilde{\widehat\varphi}:=\eta_{\code(\varphi)}\wedge \widehat\varphi.
\]
Since each $\eta_n$ is true in $\langle\N,+,\cdot,0,1,\mathcal P(\N)\rangle$, we still have
\[
\varphi\text{ is valid in all countable groups}
\quad\text{if and only if}\quad
\langle\N,+,\cdot,0,1,\mathcal P(\N)\rangle\satisfies\widetilde{\widehat\varphi}.
\]
The map $\varphi\mapsto\widetilde{\widehat\varphi}$ is primitive-recursive. It is injective on G\"odel codes of sentences because $\varphi\mapsto\code(\varphi)$ is computable, the family $(\eta_n)_{n\in\N}$ is pairwise syntactically distinct, and the left conjunct $\eta_{\code(\varphi)}$ therefore syntactically records the code of $\varphi$ inside the fixed conjunction template. By the totalization convention above, this sentence map induces a genuine total one--one reduction of modal validity for the countable-frame semantics to truth in second-order arithmetic, and in particular shows that the modal theory of countable groups under embeddings between countable groups is computably interpretable in true second-order arithmetic.
\end{proof}

\begin{corollary}\label{cor:countable-groups-reduction}
The modal theory of countable groups under embeddings between countable groups is one--one reducible to true second-order arithmetic.
\end{corollary}

\begin{proof}
This is exactly the one--one reduction induced, by the totalization convention above, by the injective sentence map $\varphi\mapsto\widetilde{\widehat\varphi}$ constructed at the end of the proof of Theorem~\ref{thm:countable-groups-second-order}.
\end{proof}

\begin{question}\label{q:second-order-sharp}
Is the modal theory of countable groups under embeddings between countable groups computably isomorphic to true second-order arithmetic?
\end{question}

\begin{question}\label{q:exact-complexity}
What is the exact computability-theoretic complexity of modal group theory?
\end{question}

\section{Propositional modal validities of groups}

We now turn from expressive power to propositional modal validities. Berger, Block, and L\"owe proved that the class of abelian groups validates \SFourTwo\ and asked whether the same is true for all groups \cite{BBL23}. We answer that question in the affirmative for the natural parameter language.

For this section, let $L$ denote a language between $\Lgrp$ and $\LgrpModal$, and write $\GrpEmb$ for the category of groups and embeddings. If $A\subseteq G$, then $L_A$ denotes the class of assertions of $L$ with parameters from $A$, that is, the sentences of the expansion obtained by adding constant symbols naming the chosen parameters. More generally, if $\mathcal L$ is a class of formulas of $\LgrpModal$, then $\mathcal L_A$ denotes the corresponding class of assertions with parameters from $A$. Under an embedding $j\colon G\to H$, the constant naming $a\in A$ is interpreted in $H$ as $j(a)$.

A propositional modal formula $\sigma(p_0,\dots,p_n)$ is \emph{valid at $G$ for substitution instances from $\mathcal L_A$} if every substitution instance $\sigma(\psi_0,\dots,\psi_n)$ with $\psi_i\in\mathcal L_A$ is true at $G$. Equivalently, these are precisely the sentential substitution instances obtained after naming the relevant parameters by constants. We write $\Val_{\GrpEmb}(G,\mathcal L_A)$ for the resulting set of propositional modal validities. Thus $\Val_{\GrpEmb}(G,L)$ refers to the sentential regime with no parameters, while $\Val_{\GrpEmb}(G,L_G)$ refers to the formulaic regime of assertions with parameters from the base group.

We shall use the standard potentialist terminology relative to a base group $G$. A \emph{world above $G$} is a group $H$ equipped with an embedding $G\to H$. An assertion $b$ is a \emph{button over $G$} if
\[
G\satisfies\necessary\possible\necessary b.
\]
At a world $H$ above $G$, the button is \emph{pushed} if $H\satisfies\necessary b$, and otherwise it is \emph{unpushed}. A finite family $b_0,\dots,b_{m-1}$ of buttons is \emph{independent over $G$} if, in every world $H$ above $G$, every desired extension of the currently pushed pattern can be realized in a further world; that is, for every subset
\[
S\supseteq\{i<m\mid H\satisfies\necessary b_i\},
\]
there is a further extension $K$ of $H$ such that
\[
K\satisfies\necessary b_i
\quad\text{if and only if}\quad
 i\in S.
\]
A finite list of assertions $d_0,\dots,d_{n-1}$ is an $n$-\emph{dial over $G$} if, in every world above $G$, exactly one of the $d_i$ holds and each dial value can be made true in a further extension. Equivalently, the partition is necessary and
\[
G\satisfies\necessary\left(\bigwedge_{i<n}\possible d_i\right).
\]
We say that $G$ admits arbitrarily long finite dials if it has such a dial of length $n$ for every $n\geq 2$. A family of buttons and a dial are \emph{independent over $G$} if, from every world above $G$, any desired extension of the currently pushed button pattern can be realized together with any desired dial value in a further world.

We shall use the following general facts from modal model theory and from the category-of-sets analysis of the broader concrete-categorical framework \cite{HW24,WSet}. Here an \emph{allowed substitution class} means a nonempty class $\mathcal L$ of formulas of $\LgrpModal$ that is closed under Boolean connectives. A span above a group $H$ means a pair of embeddings $j_0\colon H\to H_0$ and $j_1\colon H\to H_1$, and it amalgamates if there are embeddings $k_0\colon H_0\to K$ and $k_1\colon H_1\to K$ with $k_0j_0=k_1j_1$.

\begin{proposition}[General modal bounds]\label{prop:general-bounds}
Let $G$ be a group, let $A\subseteq G$, and let $\mathcal L$ be an allowed substitution class. In the category of groups and embeddings:
\begin{enumerate}[label=\textup{(\arabic*)}]
    \item if every span above every extension of $G$ amalgamates, then $\SFourTwo\subseteq \Val_{\GrpEmb}(G,\mathcal L_A)$;
    \item if $G$ admits arbitrarily long finite dials and $\mathcal L_A$ contains the corresponding dial statements, then $\Val_{\GrpEmb}(G,\mathcal L_A)\subseteq \SFive$;
    \item if $G$ admits arbitrarily large finite independent families of unpushed buttons together with arbitrarily long finite dials, with the buttons and dials mutually independent, and $\mathcal L_A$ contains the corresponding button and dial statements, then $\Val_{\GrpEmb}(G,\mathcal L_A)\subseteq \SFourTwo$.
\end{enumerate}
\end{proposition}

\subsection{The \texorpdfstring{$\SFourTwo$}{S4.2} lower bound}

\begin{theorem}\label{thm:lower-bound-S42}
Every group validates \SFourTwo\ for arbitrary assertions, even formulaic assertions with parameters in the full modal language of groups. In particular, \SFourTwo\ is valid for formulaic substitutions from every intermediate language between $\Lgrp$ and $\LgrpModal$, with parameters allowed.
\end{theorem}

\begin{proof}
The category of groups and embeddings has the amalgamation property: every span of embeddings $A\hookleftarrow C\hookrightarrow B$ has an amalgam given by the free product with amalgamation $A*_C B$. Fix a group $G$, and let $\mathcal L$ be the class of all formulas of $\LgrpModal$. Applying Proposition~\ref{prop:general-bounds}(1) with parameter set $A=G$ and with this allowed substitution class, we obtain
\[
\SFourTwo\subseteq \Val_{\GrpEmb}(G,\mathcal L_G).
\]
This is exactly the asserted full-language lower bound, and the claim for intermediate languages is immediate.
\end{proof}

\subsection{A sentential \texorpdfstring{$\SFive$}{S5} upper bound}

\begin{theorem}\label{thm:sentential-upper-bound-S5}
For every group $G$ and every language $L$ satisfying $\Lgrp\subseteq L\subseteq\LgrpModal$,
\[
\Val_{\GrpEmb}(G,L)\subseteq \SFive.
\]
In particular, the sentential propositional validities of any group are contained in \SFive.
\end{theorem}

\begin{proof}
By Proposition~\ref{prop:general-bounds}(2), applied with parameter set $A=\varnothing$ and the allowed substitution language consisting of formulas of $L$, it suffices to construct arbitrarily long finite dials. Fix $N\geq 2$. For a group $H$, let $d_1$ assert that the center of $H$ is trivial, let $d_n$ for $1<n<N$ assert that the center has size exactly $n$, and let $d_{\geq N}$ assert that the center has size at least $N$. Exactly one of these statements holds in any group.

We claim that from any group $H$ we can realize any desired dial value. To realize $d_1$, embed $H$ into $H*F_2$, whose center is trivial. To realize $d_n$ for $1<n<N$, embed $H$ into $(H*F_2)\times C_n$. Since the center of $H*F_2$ is trivial, the center of $(H*F_2)\times C_n$ is exactly $C_n$, of size $n$. Finally, $(H*F_2)\times C_N$ realizes $d_{\geq N}$. Hence every group admits arbitrarily long finite dials.
\end{proof}

\begin{theorem}\label{thm:S5-worlds}
Every group embeds into a group $N$ that validates \SFive\ for the full modal language with parameters from $N$. Equivalently, for every modal formula $\varphi(\bar x)$ of $\LgrpModal$ and every tuple $\bar a\in N$,
\[
N\satisfies \possible\necessary\varphi[\bar a]\to \varphi[\bar a].
\]
That is, every group embeds into a world satisfying the potentialist maximality principle.
\end{theorem}

\begin{proof}
This is an instance of the general theorem from modal model theory that every model of a $\forall\exists$ theory embeds into a world validating \SFive\ for the full parameter language \cite{HW24}. First-order group theory is $\forall\exists$-axiomatizable.
\end{proof}

For first-order group-theoretic assertions with parameters, this characterization is exact: the groups validating \SFive\ for substitution instances from $\Lgrp$ with parameters are precisely the existentially closed groups \cite{HW24}. It is natural to ask whether this maximality phenomenon persists beyond the purely first-order substitution regime. This leads to the following strengthened question.

\begin{question}\label{q:ec-S5}
Do all existentially closed groups validate \SFive\ for assertions in the full modal language $\LgrpModal$ with parameters?
\end{question}

\subsection{The exact formulaic modal theory}

A group has \emph{finite prime torsion spectrum} if only finitely many primes $p$ occur as orders of elements of that group. If $L$ is a language satisfying $\Lgrp\subseteq L\subseteq\LgrpModal$, we write $L_G$ for the associated class of assertions of $L$ with parameters from $G$.

\begin{theorem}\label{thm:exact-S42}
Let $G$ be a group with finite prime torsion spectrum, and let $L$ be a language satisfying $\Lgrp\subseteq L\subseteq\LgrpModal$. Then the propositional modal validities of $G$ with respect to formulaic substitution instances, that is, substitution instances by assertions with parameters from $G$, are exactly \SFourTwo. Equivalently,
\[
\Val_{\GrpEmb}(G,L_G)=\SFourTwo.
\]
\end{theorem}

\begin{proof}
Suppose $G$ has finite prime torsion spectrum. This guarantees an infinite reservoir of button primes. The lower bound is Theorem~\ref{thm:lower-bound-S42}. For the upper bound, apply Proposition~\ref{prop:general-bounds}(3) with parameter set $A=G$ and with the allowed substitution class consisting of formulas of $L$. It therefore suffices to construct arbitrarily large finite independent families of buttons together with arbitrarily long finite dials.

For a prime $p$, let $b_p$ be the sentence ``there exists an element of prime order $p$.'' From a world $H$ in the cone above $G$, push $b_p$ by embedding $H$ into $H\times B_p$, where
\[
B_p:=C_p*F_2,
\]
via $h\mapsto (h,e_{B_p})$. In a free product every torsion element is conjugate into a factor, and the center of a nontrivial free product is trivial. Hence $B_p$ has torsion of prime order exactly $p$ and trivial center. Thus $H\times B_p$ satisfies $b_p$. Since finite order is preserved by embeddings, once pushed, the button remains pushed in all further extensions.

For dials, let $N\geq 2$ be arbitrary. Define $d_1$ to say that the center is trivial, define $d_n$ for $1<n<N$ to say that the center has size exactly $n$, and let $d_{\geq N}$ say that the center has size at least $N$. Exactly one of $d_1,\dots,d_{N-1},d_{\geq N}$ holds in any group. From any world $H$, these values can be realized by the embeddings
\[
H\hookrightarrow H*F_2,\qquad H\hookrightarrow (H*F_2)\times C_n\ \text{ for } 1<n<N,\qquad H\hookrightarrow (H*F_2)\times C_N.
\]
Hence dials of every finite length $N\geq 2$ exist.

Now fix $m\in\N$ and choose distinct primes $p_1,\dots,p_m$ greater than $N$ such that $G$ has no element of order $p_i$ for any $i$. Let $P=\{p_1,\dots,p_m\}$.

The dial moves do not affect the buttons $b_p$ for $p\in P$. Indeed, $F_2$ is torsion-free, and in a free product every torsion element is conjugate into a factor, so $H*F_2$ introduces no new torsion beyond the torsion already present in $H$. Likewise, $(H*F_2)\times C_n$ introduces, beyond the primes already present in $H$, only primes dividing $n$, and $(H*F_2)\times C_N$ introduces, beyond the primes already present in $H$, only primes dividing $N$. Since each $p\in P$ is greater than $N$, no dial move can create an element of order $p$ if it was not already present.

Conversely, pushing a button does not affect the dial. Since $Z(B_p)=1$,
\[
Z(H\times B_p)=Z(H)\times Z(B_p)=Z(H)\times 1,
\]
so the size of the center is unchanged when we pass from $H$ to $H\times B_p$. Moreover, if $q\neq p$, then $B_p=C_p*F_2$ has no elements of order $q$, so passing from $H$ to $H\times B_p$ does not create any new elements of order $q$. Hence pushing $b_p$ does not affect any other button $b_q$ with $q\neq p$.

Thus the buttons $\{b_p\mid p\in P\}$ are mutually independent and independent of the dial. Indeed, given any world $H$ in the cone above $G$, let $S\subseteq P$ be any desired button pattern extending the buttons already pushed in $H$, and let
\[
T:=\{p\in S\mid H\satisfies\neg b_p\}
\]
be the additional buttons to be pushed. Given also any target dial value among $d_1,\dots,d_{N-1},d_{\geq N}$, first push all $b_p$ for $p\in T$ via
\[
H\hookrightarrow H\times \prod_{p\in T} B_p.
\]
Let
\[
K:=H\times \prod_{p\in T} B_p.
\]
Then set the dial from $K$ by one of the embeddings
\[
K\hookrightarrow K*F_2,\qquad K\hookrightarrow (K*F_2)\times C_n\ \text{ for } 1<n<N,\qquad K\hookrightarrow (K*F_2)\times C_N.
\]
All these maps are canonical factor embeddings and hence injective, so finite orders are preserved throughout. Since $m$ and $N$ were arbitrary, Proposition~\ref{prop:general-bounds}(3) yields
\[
\Val_{\GrpEmb}(G,L_G)\subseteq \SFourTwo.
\]
Combining this with the lower bound gives equality.
\end{proof}

\begin{corollary}\label{cor:category-S42}
Let $L$ be a language satisfying $\Lgrp\subseteq L\subseteq\LgrpModal$. Then the exact propositional modal theory of the category of groups under embeddings, with formulaic substitutions from $L$ (equivalently, with assertions in $L$ and parameters allowed), is \SFourTwo.
\end{corollary}

\begin{proof}
The lower bound follows from Theorem~\ref{thm:lower-bound-S42}. For the upper bound, apply Theorem~\ref{thm:exact-S42} to any torsion-free group.
\end{proof}

\markboth{MODAL GROUP THEORY}{MODAL GROUP THEORY}
\printbibliography[heading=mgtbibliography]

@article{BBL23,
  author  = {Berger, S{\"o}ren and Block, Alexander Christensen and L{\"o}we, Benedikt},
  title   = {The modal logic of abelian groups},
  journal = {Algebra Universalis},
  volume  = {84},
  year    = {2023},
  note    = {Article 25},
  doi     = {10.1007/s00012-023-00821-9}
}

@book{ChangKeisler,
  author    = {Chang, C. C. and Keisler, H. J.},
  title     = {Model Theory},
  publisher = {North-Holland},
  address   = {Amsterdam},
  edition   = {3},
  year      = {1990}
}

@article{HNN49,
  author  = {Higman, Graham and Neumann, B. H. and Neumann, Hanna},
  title   = {Embedding theorems for groups},
  journal = {Journal of the London Mathematical Society},
  volume  = {24},
  number  = {4},
  year    = {1949},
  pages   = {247--254},
  doi     = {10.1112/jlms/s1-24.4.247}
}

@incollection{Hodges84,
  author    = {Hodges, Wilfrid},
  title     = {Finite extensions of finite groups},
  editor    = {M{\"u}ller, Gert H. and Richter, Michael M.},
  booktitle = {Models and Sets},
  series    = {Lecture Notes in Mathematics},
  volume    = {1103},
  publisher = {Springer},
  address   = {Berlin},
  year      = {1984},
  pages     = {193--206},
  doi       = {10.1007/BFb0099387}
}

@article{HW24,
  author  = {Hamkins, Joel David and Wo{\l}oszyn, Wojciech Aleksander},
  title   = {Modal model theory},
  journal = {Notre Dame Journal of Formal Logic},
  volume  = {65},
  number  = {1},
  year    = {2024},
  pages   = {1--37},
  doi     = {10.1215/00294527-2024-0001}
}

@article{HamkinsArithmeticPotentialism,
  author  = {Hamkins, Joel David},
  title   = {The modal logic of arithmetic potentialism and the universal algorithm},
  journal = {Philosophia Mathematica},
  volume  = {34},
  number  = {1},
  year    = {2026},
  pages   = {137--182},
  doi     = {10.1093/philmat/nkag001}
}

@article{HamkinsLinneboPotentialism,
  author  = {Hamkins, Joel David and Linnebo, {\O}ystein},
  title   = {The modal logic of set-theoretic potentialism and the potentialist maximality principles},
  journal = {The Review of Symbolic Logic},
  volume  = {15},
  number  = {1},
  year    = {2022},
  pages   = {1--35},
  doi     = {10.1017/S1755020318000242}
}

@book{LyndonSchupp,
  author    = {Lyndon, Roger C. and Schupp, Paul E.},
  title     = {Combinatorial Group Theory},
  publisher = {Springer},
  address   = {Berlin},
  year      = {1977}
}

@article{Myhill55,
  author  = {Myhill, John},
  title   = {Creative sets},
  journal = {Zeitschrift f{\"u}r mathematische Logik und Grundlagen der Mathematik},
  volume  = {1},
  number  = {2},
  year    = {1955},
  pages   = {97--108},
  doi     = {10.1002/malq.19550010205}
}

@misc{WoloszynHomomorphisms,
  author = {Wo{\l}oszyn, Wojciech Aleksander},
  title  = {Modal group theory: homomorphisms},
  year   = {2025},
  sorttitle = {Modal group theory 1 homomorphisms},
  note   = {\mbox{e-mail} circulated preprint}
}

@misc{WoloszynEpimorphisms,
  author = {Wo{\l}oszyn, Wojciech Aleksander},
  title  = {Modal group theory: epimorphisms},
  year   = {2025},
  sorttitle = {Modal group theory 2 epimorphisms},
  note   = {\mbox{e-mail} circulated preprint}
}

@misc{WoloszynLinearOrders,
  author = {Wo{\l}oszyn, Wojciech Aleksander},
  title  = {The modal theory of linear orders},
  year   = {2025},
  sorttitle = {Modal theory of linear orders},
  note   = {\mbox{e-mail} circulated preprint}
}

@misc{WSet,
  author      = {Wo{\l}oszyn, Wojciech Aleksander},
  title       = {The modal theory of the category of sets},
  year        = {2026},
  eprinttype  = {arxiv},
  eprint      = {2603.25550},
  doi         = {10.48550/arXiv.2603.25550}
}

@incollection{Ziegler80,
  author    = {Ziegler, Martin},
  title     = {Algebraisch abgeschlossene Gruppen},
  editor    = {Adian, S. I. and Boone, W. W. and Higman, Graham},
  booktitle = {Word Problems II: The Oxford Book},
  series    = {Studies in Logic and the Foundations of Mathematics},
  volume    = {95},
  publisher = {North-Holland},
  address   = {Amsterdam},
  year      = {1980},
  pages     = {449--576}
}

\end{document}